\documentclass[11pt]{amsart}

\usepackage{amsmath,amssymb,amsfonts,amsthm}
\usepackage{aliascnt}
\usepackage{bbm}
\usepackage{mathtools}
\usepackage{mathrsfs}
\usepackage{yhmath}

\usepackage[margin=1in]{geometry}
\usepackage[shortlabels]{enumitem}
\setlist{listparindent=\parindent, parsep=0pt}

\usepackage{comment}
\usepackage[dvipsnames]{xcolor}
\colorlet{RED}{red}
\usepackage{hyperref}
\hypersetup{linktocpage,colorlinks,linkcolor=blue,citecolor=magenta}
\usepackage[capitalise]{cleveref}
\crefformat{equation}{(#2#1#3)}

\usepackage{cite}

\theoremstyle{plain}
\newtheorem{thm}{Theorem}[section]
\newaliascnt{prop}{thm}
\newtheorem{prop}[prop]{Proposition}
\aliascntresetthe{prop}
\newtheorem{lemma}{Lemma}[section]

\theoremstyle{definition}
\newaliascnt{mydef}{thm}

\aliascntresetthe{mydef}
\newaliascnt{remark}{thm}
\newtheorem{remark}[remark]{Remark}
\aliascntresetthe{remark}
\crefname{thm}{Theorem}{Theorems}
\Crefname{thm}{Theorem}{Theorems}
\crefname{prop}{Proposition}{Propositions}
\Crefname{prop}{Proposition}{Propositions}
\crefname{lemma}{Lemma}{Lemmas}
\Crefname{lemma}{Lemma}{Lemmas}
\crefname{corollary}{Corollary}{Corollaries}
\Crefname{corollary}{Corollary}{Corollaries}
\crefname{mydef}{Definition}{Definitions}
\Crefname{mydef}{Definition}{Definitions}
\crefname{remark}{Remark}{Remarks}
\Crefname{remark}{Remark}{Remarks}

\numberwithin{equation}{section}

\DeclareMathOperator{\sgn}{sgn}

\newcommand{\Z}{\mathbb{Z}}
\newcommand{\R}{\mathbb{R}}
\newcommand{\T}{\mathbb{T}}

\newcommand{\Pc}{\mathcal{P}}
\newcommand{\cI}{\mathcal{I}_+}
\newcommand{\indic}{\mathbbm{1}}

\newcommand{\DO}{\mathrm{DO}}

\newcommand{\wh}{\widehat}

\newcommand{\Hrel}[2]{H(#1\mid#2)}

\title[Phase Transitions in Multimodal Models]{Phase transitions in Doi--Onsager, noisy transformer, and other multimodal models}

\author[K. Mun]{Kyunghoo Mun}
\address{Kyunghoo Mun, Carnegie Mellon University, Department of Mathematical Sciences, Pittsburgh, PA}
\email{kmun@andrew.cmu.edu}

\author[M. Rosenzweig]{Matthew Rosenzweig}
\address{Matthew Rosenzweig, Carnegie Mellon University, Department of Mathematical Sciences, Pittsburgh, PA}
\email{mrosenz2@andrew.cmu.edu}
\thanks{M.R. was supported by NSF grants DMS-2441170, DMS-2345533, DMS-2342349.}

\begin{document}

\begin{abstract}
We study phase transitions for repulsive-attractive mean-field free energies on the circle. For a $\frac{1}{n+1}$-periodic interaction whose Fourier coefficients satisfy a certain decay condition, we prove that the critical coupling strength $K_c$ coincides with the linear stability threshold $K_\#$ of the uniform distribution and that the phase transition is continuous, in the sense that the uniform distribution is the unique global minimizer at criticality. The proof is based on a sharp coercivity estimate for the free energy obtained from the constrained Lebedev--Milin inequality. 

We apply this result to three motivating models for which the exact value of the phase transition and its (dis)continuity in terms of the model parameters was not fully known. For the two-dimensional \emph{Doi--Onsager} model $W(\theta)=-|\sin(2\pi\theta)|$, we prove that the phase transition is continuous at $K_c=K_\#=3\pi/4$. For the \emph{noisy transformer} model $W_\beta(\theta)=(e^{\beta\cos(2\pi\theta)}-1)/\beta$, we identify the sharp threshold $\beta_*$ such that $K_c(\beta) = K_\#(\beta)$ and the phase transition is continuous for $\beta \leq \beta_*$, while $K_c(\beta)<K_\#(\beta)$ and the phase transition is discontinuous for $\beta > \beta_*$. We also obtain the corresponding sharp dichotomy for the noisy \emph{Hegselmann--Krause} model $W_{R}(\theta) = (R-2\pi|\theta|)_{+}^2$ . 
\end{abstract}

\maketitle

\section{Introduction}\label{sec:intro}
In this note, we consider the minimization of the \emph{free energy}
\begin{align}\label{eq:free-energy}
\mathcal{F}_K(q)
:=\int_\T \log {q} \, dq(\theta) \;-
K\iint_{\T \times \T} W(\theta - \theta')\, dq(\theta) dq(\theta')
\end{align}
defined on Borel probability measures $q \in \mathcal{P}(\T)$ on the unit circle, identified as $\T\coloneqq [-\frac12,\frac12)$ with periodic boundary conditions. Here, $K\ge 0$ is the \emph{coupling constant/interaction strength}, and $W$ is the \emph{interaction potential} assumed to be real-valued and even.
If $K=0$, the free energy is exactly the relative entropy $\Hrel{q}{q_u} \coloneqq \int_{\T} \log (q/q_u)  \, dq$, which is uniquely minimized by the \emph{uniform distribution} $q_u \coloneqq 1$. For an interaction with an attractive part, by which we mean $W$ has at least one positive Fourier coefficient, as $K$ increases, a new nonuniform minimizer appears due to the dominance of the interaction energy. This change in behavior is called a \textit{phase transition}, which has been mathematically studied for instance in \cite{chayesMcKeanVlasovEquationFinite2010,carrilloLongtimeBehaviourPhase2020,delgadino2023phase} and numerous other works. More precisely, there exists a critical coupling $K_c>0$ such that\footnote{Note that this definition of phase transition refers to uniqueness of the global minimizer, but not necessarily uniqueness of critical points. This is an important distinction, which we will come back to later below. Since $q_u$ is always a critical point, nonuniqueness of critical points when $K>K_c$ always holds; but in principle, there could be a non-minimizing critical point when $K<K_c$. In the present work, a critical point is understood as a solution of the Euler--Lagrange/Kirkwood--Monroe equation associated to $\mathcal{F}_K$. In general, there are different notions of critical points (see, e.g., \cite{delgadino2023phase} for more discussion on this topic), but in our setting, they all coincide.}
\begin{enumerate} \label{phase transition def}
    \item \label{item: subcritical} in the \emph{subcritical} regime $K< K_{c}$, $q_u$ is the unique minimizer of free energy $\mathcal{F}_K$;
    \item \label{item: critical} in the \emph{critical} regime $K=K_{c}$, $q_{u}$ remains a minimizer of $\mathcal{F}_K$ {(though a priori not necessarily unique)};
    \item \label{item: supercritical} in the \emph{supercritical} regime $K>K_{c}$, there exists a nonuniform minimizer $q \neq q_{u}$ of $\mathcal{F}_K$.
\end{enumerate}
The phase transition is closely related to the stability of the uniform state. We therefore define the \textit{linear stability threshold} $K_\#$ as the value at which the second variation of $\mathcal{F}_K$ at $q_u$ loses positivity (see \eqref{eq:Ksharp} and more generally \cref{sec:contipt}). In general, one has $K_c \leq K_\#$, and we seek conditions that guarantee equality.

Given the existence of a phase transition, an important problem is to determine the exact value of $K_{c}$. This question was studied in \cite{chayesMcKeanVlasovEquationFinite2010,carrilloLongtimeBehaviourPhase2020} and following works by means of the {linear stability threshold} $K_\#$ at $q_u$ (see \cref{phase trans continuity prop} below). It turns out that this problem is closely related to the \textit{continuity of the phase transition}; a phase transition is said to be \textit{continuous} at the critical interaction strength $K_{c},$ if a global minimizer $q_{K_c}$ is unique and for any family of global minimizers $\{q_{K}: K > K_{c}\}$,
\begin{equation}\label{eq:ptcont_def}
        \limsup_{K \to K_{c} +} \|q_{K} - q_{\operatorname{unif}} \|_{L^1} = 0.
\end{equation}
Otherwise, the phase transition is \textit{discontinuous}.  Despite the comparisons made in previous works, a complete characterization of when $K_{c} = K_{\#}$ remains open for multimodal interaction potentials $W$, with our recent work \cite{mun_phase_2026} providing such a characterization for bimodal interactions. 

From a dynamical perspective, critical points of \eqref{eq:free-energy} are equivalent to the stationary solutions of the \textit{McKean--Vlasov} PDE  
\begin{align} \label{MV equation}
\partial_{t} q_{t} = \frac{1}{2} \partial_{\theta}^{2} q_{t}- K \partial_{\theta}  \big(q_t \ \partial_{\theta}(W * q_{t}) \big) = - \frac12\partial_{\theta} \left[ q_t \partial_{\theta} \Big( \frac{\delta \mathcal{F}}{\delta q}(q_t)  \Big)\right],
\end{align}
which is the $2$-Wasserstein gradient flow of $\mathcal{F}_K$. Equation \eqref{MV equation} can also be interpreted as a \emph{mean-field limit} for a system of $N$ interacting particles given by
\begin{align}\label{Dynamics}
d\theta_{i}=  \frac{K}{N} \sum_{j=1}^{N}  W'(\theta_i - \theta_j) \ dt + d B_{i}, \qquad  i\in[N],
\end{align}
where $(B_i)_{i=1}^{N}$ are i.i.d. Brownian motions on $\T$. The mean-field limit refers to the convergence, as $N \to \infty$, of the {\it empirical measure} $q_{N,t}\coloneqq \frac1N \sum_{i=1}^N \delta_{\theta_i(t)}$
associated to a solution $\Theta_N(t) \coloneqq (\theta_1(t), \dots, \theta_N(t))$ of the system \eqref{Dynamics}. Assuming the initial points $\Theta_N(0)$ are such that $q_{N,0} \underset{N\to \infty}{\longrightarrow} q_{0}$, an application of It\^o's rule leads one to expect that, for $t>0$, $q_{N,t} \underset{N \to \infty}{\longrightarrow} q_t$, where $q_t$ solves the macroscopic \textit{mean-field equation} \eqref{MV equation}. For this reason, it is common to refer to $\mathcal{F}_K$ as the \emph{mean-field/macroscopic free energy} to contrast with the \emph{microscopic free energy} associated to the $N$-particle system \eqref{Dynamics}.

\subsection{Multimodal Interactions}\label{ssec:intro-examples}
Mean-field systems of the above form appear in diverse contexts. Here, we discuss three such examples for which the mathematical study of \eqref{eq:free-energy} and of phase transitions remains limited.

The \emph{Doi--Onsager model (or Onsager model)} describes the orientational statistics of a spatially homogeneous suspension of long, rigid rods by a probability density $q=q(\theta)$ on the circle $\mathbb{S}^{1}$, where $\theta$ parametrizes the rod direction.
In a mean-field approximation, excluded-volume interactions are encoded by an even kernel
\begin{equation} \label{eq:Onsager kernel}
    W_{\DO}(\theta)=-|\sin(2\pi\theta)|=-\frac{2}{\pi}+\frac{4}{\pi}\sum_{\ell=1}^{\infty}\frac{\cos(4\pi\ell\theta)}{4\ell^2-1},
\end{equation}
and equilibrium states are characterized as minimizers (or, more generally,
critical points) of a free energy functional \eqref{eq:free-energy}; see Onsager's classical work \cite{onsagerEffectsShapeColloidal1949} and the Doi--Edwards framework \cite{doiEdwardsTheoryPolymerDynamics1986}. In \cite{symmetricsolutionsonsager}, the authors showed that all stationary solutions {$q$} must be axially symmetric. That is, $q(\theta_0 - \theta) = q(\theta_0+\theta)$ for some $\theta_0$, and there exists a family $(q_{n})_{n=1}^{\infty} \subset \mathcal{P}(\T)$, with corresponding coupling constants $(K_n)_{n=1}^{\infty} \subset (0, \infty)$, such that $q_{n}$ is $\frac{1}{2n}$-periodic and is a critical point of the free energy at $K = K_n$. \cite{upperKonsager} provided a lower bound for the critical interaction strength for a general class of {Lipschitz} interaction potentials $W$, showing that $K_{c} \geq {\pi/4}$ for the Onsager model. \cite{upperonsager2} refined this estimate to $K_{c} \geq \sqrt{\pi}/2$. However, the exact value of $K_c$ for the Doi--Onsager model has remained to our knowledge open, while \cite{transformertry} showed that the transition for the Doi--Onsager model on the hypersphere $\mathbb{S}^{d-1}$ for $d\ge 3$ is discontinuous. See \eqref{eq:ptcont_def} for the definition of continuity of the phase transition.

A more recent example is the \emph{(noisy) transformer model.} Despite the empirical success of large language models, the mathematical underpinnings of the underlying transformer architectures \cite{attention} are not fully understood. Geshkovski et al. \cite{transformer} proposed a surrogate model of self-attention mechanisms, {interpreting them as a mean-field interacting particle system on ${\mathbb{S}}^{d-1}$ for $d\ge 2$. In the special $d=2$ case, the setting of the model reduces to $\mathbb{T}$ with the interaction potential}
\begin{align} \label{transformer interaction}
    W_{\beta}(\theta) =  \frac{ e^{\beta \cos (2\pi\theta)} -1}{\beta} = \frac{I_{0}(\beta)-1} {\beta} + \frac{2}{\beta} \sum_{n=1}^{\infty} I_{n}(\beta) \cos (2\pi n\theta), 
\end{align}
where $I_{n}(\cdot)$ denotes the $n$-th modified Bessel function of the first kind and $\beta >0$ is a model parameter interpreted as inverse temperature. 
Shalova and Schlichting \cite{transformertry} further analyzed the noisy transformer dynamics associated to \eqref{transformer interaction}, along with its higher-dimensional counterparts. They computed $K_{\#}(\beta)=\beta / (2I_{1}(\beta))$ and identified bifurcations at $q_{u}$ for specific values of $K$, showing that the first bifurcation occurs at $K_{\#}(\beta)$ via a linear stability analysis. {Recently, Balasubramanian et al. \cite{rigollet25} showed that there exist $\beta_{+} > \beta_{-} >0$ such that the phase transition is continuous at $K_{c}(\beta) = K_{\#}(\beta)$ for $\beta \le \beta_{-}$ and is discontinuous at $K_{c} < K_{\#}$ for $\beta > \beta_{+}$. Moreover, they showed the estimate $\beta_{+}\le \beta_*$, where $\beta_* \approx 2.447$ is the unique solution in $\R_+$ of $I_{1}(\beta) = 2I_{2}(\beta)$.\footnote{Since $\beta \mapsto \frac{I_{2}(\beta)}{I_{1}(\beta)}$ is strictly increasing on $(0, \infty)$, the solution $\beta_{*}$ of $I_2(\beta) = \tfrac12 I_1(\beta)$ is unique.}  This value was shown to be the threshold at which the first bifurcation point $K_{\#}(\beta)$ changes from supercritical to subcritical.}

Lastly, we discuss the \emph{(noisy) Hegselmann--Krause model}, a prototypical bounded-confidence model in opinion dynamics. In the original Hegselmann--Krause model \cite{hegselmannOpinionDynamicsBounded2002}, agents update their opinions by averaging the opinions of their neighbors within a confidence radius $R\in [0, \pi]$. Beyond this discrete-time agent-based formulation, continuum and mean-field descriptions of bounded-confidence dynamics have also been studied mathematically. In particular, Wang et al. \cite{wangNoisyHegselmannKrauseSystems2017} analyzed a continuous-time noisy Hegselmann--Krause system through its mean-field Fokker--Planck equation, and numerically identified an order--disorder phase transition, although identification of $K_c$ was not accomplished. After reducing to a one-dimensional interaction on the normalized torus, we consider the free energy \eqref{eq:free-energy} with interaction
\begin{equation} \label{eq: HK interaction}
W_{R}(\theta) = \left( {R} - 2\pi |\theta| \right)_{+}^{2} = \frac{R^{3}}{3\pi} + \frac{4}{\pi} \sum_{\ell=1}^{\infty} \frac{R\ell - \sin R\ell}{\ell^{3}} \cos(2\pi\ell \theta).
\end{equation}
\cite{carrilloLongtimeBehaviourPhase2020} showed a discontinuous transition at $K_c < K_\#$ for sufficiently small $R$. Moreover, \cite{transformertry} showed a discontinuous transition at $K_c < K_\#$ on $\mathbb{S}^{d-1}, d \geq 3$, but only for a specific choice of $R$. 

\subsection{Main results}\label{ssec:intro-mr}
Motivated by the examples of the previous subsection, we study a general multimodal, repulsive-attractive interaction of the form
\begin{align} \label{eq: W-multimodal}
    W(\theta) = 2\sum_{k\in \cI} \wh W(k) \cos (2\pi k \theta) + 2\sum_{k\in \mathcal{I}_-} \wh W(k) \cos (2\pi k \theta) \eqqcolon W^{+}(\theta) - W^{-}(\theta).
\end{align}
Here, $\mathcal{I}_{\pm} \coloneqq \{k\in \mathbb{N}:\pm\wh W(k)>0\}$, $W^{\pm}$ are the positive- and negative-definite parts of $W$, and we have safely discarded the zeroth Fourier mode (i.e., assumed $W$ to be zero average) because it amounts to adding a constant to the free energy $\mathcal{F}_{K}$. The existence of a phase transition as defined above at some $K_c \ge 0$ is equivalent to $\mathcal{I}_{+}\neq \emptyset$ (see, e.g., \cite{vanderWaals}). In the $\frac{1}{n+1}$-periodic setting below, the first active attractive mode is $k=n+1$. Since replacing $W$ by $W/(2\wh W(n+1))$ simply rescales the coupling constant from $K$ to $2K\wh W(n+1)$, we may always reduce to the case $2\wh W(n+1)=1$ and will impose this normalization in our main result \cref{thm:main pt} below.

We assume that $W\in \mathcal{D}'(\T)$ is given by a Borel measurable function.\footnote{In fact, one can allow for the negative-definite part $W^{-}$ to be an arbitrary tempered distribution, provided that one interprets the convolution in the distributional sense.} Decomposing $W^{+} = W_{+}^{+} - W_{-}^{+}$ into its positive and negative parts, we assume that $W_{+}^{+}$ has the following property: for any $K\ge 0$, there exists $C_K\ge 0$ such that
\begin{align}\label{eq:W-assumption}
    H(q\mid q_u) - K \iint_{\T \times \T} W_{+}^{+}(\theta - \theta')\, dq(\theta) dq(\theta') \geq -C_K, \quad \forall q \in \mathcal{P}(\T).
\end{align}
Obviously, the condition \eqref{eq:W-assumption} is satisfied if $W_{+}^{+} \in L^\infty(\T)$. This assumption ensures that $\mathcal{F}_K$ is bounded from below for any $K\ge 0$. Initially defining $\mathcal{F}_K$ on $C^\infty(\T)\cap\mathcal{P}(\T)$, the reader may check that it has a unique continuation to a lower semicontinuous functional $\mathcal{P}(\T)\rightarrow (-\infty,\infty]$ and
\begin{align}
    \mathcal{F}_K(q) < \infty \Longleftrightarrow H(q\mid q_u) + \int_{\T^2} \left(W_{-}^{+}+W^{-}\right)(\theta - \theta') \, dq(\theta) dq(\theta') < \infty.
\end{align}
In particular, if $\mathcal{F}_K(q) < \infty$ then $q$ must be absolutely continuous with respect to the Lebesgue measure.\footnote{Here and throughout this paper, we abuse notation by using the same symbol for both a measure and its Lebesgue density (when absolutely continuous).}
Moreover, by the Dunford-Pettis theorem, a sequence $\{q_n\in \mathcal{P}(\T): \mathcal{F}_K(q_n) \leq C\}$ is weakly relatively compact with respect to the $\sigma(L^1, L^\infty)$-topology, and thus a minimizer of $\mathcal{F}_K$ always exists. We further assume that the interaction energy is upper semicontinuous.  
 This ensures that the uniqueness of the global minimizer at $K_c$ implies the convergence \eqref{eq:ptcont_def}. Thus, continuity of the transition is equivalent to the uniqueness of minimizers in the critical regime $K=K_c$. 

In principle, one could consider more singular interactions, such as the circular log gas interaction $W(\theta) = -\log|2\sin(\pi\theta)|$, but the notion of phase transition must be modified to account for the fact that the free energy is not lower-bounded for all $K\ge 0$ in this case. Indeed, for $K>1$, the free energy is not lower-bounded, and thus no minimizer exists; however, one still may have critical points (see \cref{rem:log-gas} below).


The primary contribution of this note is to determine the exact value of $K_{c}$ and the conditions under which the phase transition is (dis)continuous, as well as to give sufficient conditions for the stronger assertion of uniqueness of critical points. This is the content of the following theorem, which is our main result and whose proof is presented in \cref{sec:contipt}. In \cref{ssec:intro-applications}, we will apply this theorem to the three motivating examples of the previous subsection. The normalization assumption on $\wh{W}(n+1)$ is harmless by rescaling, as explained above. We defer comments on the other assumptions imposed in \cref{thm:main pt} to the discussion of the proof in the next subsection.

\begin{thm} \label{thm:main pt}
    For an integer $n\ge 0$, suppose that $W$ as above is $\frac{1}{n+1}$-periodic,  $2\wh W(n+1)=1$, and $\wh W$ satisfies the decay condition
    \begin{align} \label{eq:decay condition}
        2\wh W(k) \leq \frac{n+1}{k}, \quad \forall k\ge 1.
    \end{align}
    Then a continuous phase transition occurs at the critical coupling strength $K_{c} = 1 = K_{\#}.$ Furthermore, if $K\le \frac12$, then $q_u$ is the unique critical point of $\mathcal{F}_K$.
\end{thm}

\cref{thm:main pt} asserts that $\mathcal{F}_K$ has a unique global minimizer $q_u$ if and only if $K\le 1$. In particular, if $K>1$, then we have nonuniqueness of critical points of $\mathcal{F}_K$. On the other hand, if $K\le 1/2$, the theorem asserts that $q_u$ is the unique critical point of $\mathcal{F}_K$. Thus, there remains a gap $K\in (1/2,1)$ in which we have uniqueness of the global minimizer but do not know whether we also have uniqueness of critical points. This gap reflects a limitation of our method (see the discussion of the proof below), and it would be interesting to determine whether uniqueness of critical points holds for all $K\le 1$, which we tentatively conjecture to be the case.

\begin{remark}
    We emphasize that \cref{thm:main pt} makes no quantitative assumptions on the repulsive part $W^{-}$ of the interaction because $\wh{W}(k)\le 0$ for all $k\in \mathcal{I}_-$, and thus the condition \eqref{eq:decay condition} is automatically satisfied for such $k$. In particular, the repulsive part can be arbitrarily strong, and the phase transition is still continuous at $K_c = 1$. 

\end{remark}

\begin{remark}\label{rem:log-gas}
    The interaction $W(\theta) = -\log|2\sin(\pi\theta)|$ is a special case of \eqref{eq: W-multimodal} with $\cI = \mathbb{N}$ and $\wh W(k) = 1/(2k)$ for all $k \in \mathbb{N}$ (i.e., the condition \eqref{eq:decay condition} holds for $n=1$, with equality for all $k$). This is the one-dimensional (circular) attractive log gas, which was studied by Chodron de Courcel, Serfaty, and the second author in the general $d$-dimensional setting in \cite{chodrondecourcelAttractiveLogGas2025}, and whose repulsive counterpart is an important model in random matrix theory \cite{forresterLoggasesRandomMatrices2010}. As noted above, the free energy is bounded from below if and only if $K\le 1$ and thus for $K>1$, there does not exist a minimizer. However, the uniform distribution $q_u$ is still a critical point of $\mathcal{F}_K$ for all $K$. 
    This model turns out to be exactly solvable and essentially everything can be said about both its equilibrium and nonequilibrium properties.

    As first observed by C\'{e}pa and L\'{e}pingle \cite{cepaBrownianParticlesElectrostatic2001} in passing in work on the well-posedness of measure-valued solutions, the Fourier coefficients of a solution $q_t$ of \eqref{MV equation} satisfy the coupled but finitely closed system of differential equations 
    \begin{align}
       2\frac{d}{dt}\wh{q}_t(k) =  -4\pi^2 k\left(k - K\right)\wh q_t(k) + 4\pi^2k K\sum_{j=1}^{k-1}\wh{q}_t(j)\wh q_t(k-j), \qquad k\ge 1,
    \end{align}
    where the $k$-th Fourier coefficient only depends on the first $k$ Fourier coefficients.
    This remarkably simplified structure is a consequence of the special form $\wh{\partial_\theta W}(k) = - i\pi\sgn(k)$. Setting this equation equal to zero allows one to fully characterize the stationary solutions of \eqref{MV equation}, equivalently, critical points of $\mathcal{F}_K$. Namely, for $K\notin \Z_{+}$, $q_u$ is the unique critical point of $\mathcal{F}_K$, while for $K=n\in \Z_{+}$, the critical points are given by translates of
    \begin{align}
        \frac{1-c^2}{1+c^2-2c\cos(2\pi n \theta)}, \quad c\in [0,1) \qquad \text{or} \qquad \frac{1}{n}\sum_{i=0}^{n-1} \delta_{\frac{i}{n}}(\theta).
    \end{align}
    Additionally, one can establish the well-posedness of measure-valued solutions under optimal assumptions and fully characterize the dynamical stability properties of these critical points. These results are the subject of forthcoming work \cite{decourcelExactSolvabilityMacroscopic} by Chodron de Courcel, Serfaty, and the second author.
\end{remark}

\begin{remark}\label{rem:decay-condition}
    Our assumptions on $W$ that ensure that $\mathcal{F}_K$ is bounded from below for any $K\ge 0$ imply that there is some $k\ge 1$ such that the inequality in \eqref{eq:decay condition} is strict. Indeed, otherwise, $W\propto \log|2\sin(\pi\theta)|$ and thus $\mathcal{F}_K$ is not bounded from below for $K>1$, which is a contradiction.
\end{remark}

\subsection{Comments on the proof}\label{ssec:intro-proof}

The key ingredient to the proof of \cref{thm:main pt} is a sharp functional inequality relating the entropy to the $\dot{H}^{-1/2}$ seminorm (see \cref{prop:sharp-ineq} below), which allows to prove coercivity of the free energy $\mathcal{F}_K$. This inequality is a consequence---or more precisely, the dual form---of what is known in the harmonic analysis literature as the \emph{(constrained) Lebedev--Milin inequality}, presented in \cref{prop:ops-widom} below. The original\footnote{The original article is in Russian. See \cite[Chapter 5]{duren2001univalent} for a more modern presentation in English.} inequality of Lebedev and Milin \cite{lebedevMilinOneInequality1965} corresponds to the $n=0$ case and the constrained $n=1$ version was later shown by Osgood et al. \cite{osgoodExtremalsDeterminantsLaplacians1988}. Subsequently, Widom \cite{widomInequalityOsgoodPhillips1988} elegantly observed that the inequality of \cite{osgoodExtremalsDeterminantsLaplacians1988} is the first in a family of constrained inequalities imposing higher-order vanishing of the Fourier coefficients, which are a straightforward consequence of the Szeg\"o limit theorem \cite{MR890515}. We refer to \cite{changSharpInequalityExponentiation2023} and references therein for further discussion, as well as generalizations.

\begin{prop}\label{prop:ops-widom}
Let $\phi:\T\rightarrow\R$, and let $\Phi$ be its harmonic extension to the unit disk $D\subset\mathbb{R}^2$. For an integer $n\ge 0$, if $\int_{\T} e^{\phi(\theta)} e^{2\pi i k \theta} \, d\theta=0 $ for all $1 \leq k \leq n$,
then
\begin{equation}\label{eq:ops-widom}
\log\left(\int_{\T}e^{\phi(\theta)}\,d\theta\right)
- \int_{\T}\phi \,d\theta \le\ \frac{1}{4(n+1)\pi}\int_D |\nabla \Phi|^2\,dx\,dy.
\end{equation}
Moreover, equality holds if and only if $e^{-\phi}$ is a trigonometric polynomial of degree at most $n+1$. 

\end{prop}

\begin{remark}\label{rem:H12}
    The right-hand side of \eqref{eq:ops-widom} can be expressed in terms of the $\dot{H}^{\frac12}$-seminorm of $\phi$ as
    \[
    \frac{1}{4(n+1)\pi}\int_D |\nabla \Phi|^2\,dx\,dy = \frac{1}{4\pi (n+1)}\sum_{k\in\Z} |2\pi k| |\wh{\phi}(k)|^2 = \frac{1}{n+1}\sum_{k=1}^\infty |k||\wh{\phi}(k)|^2.
    \]

\end{remark}

\begin{remark}
    At the risk of being obvious, let us point out that a $\frac{1}{n+1}$-periodic function $\phi$ necessarily satisfies the vanishing conditions $\int_{\T} e^{\phi(\theta)} e^{2\pi i k \theta} \, d\theta=0 $ for all $1 \leq k \leq n$.
\end{remark}

The strategy is to use this dual form \cref{prop:sharp-ineq} to control the interaction energy term in $\mathcal{F}_{K}(q)$ by $\|q\|_{\dot{H}^{-\frac12}}^2$ to establish a coercivity estimate for the free energy. This is done term-by-term in the Parseval expansion of both the interaction energy and  $\|q\|_{\dot{H}^{-\frac12}}^2$, which allows us to conclude that
\begin{align}
    K \le \frac{n+1}{2\sup_{k\in \mathcal{I}_+} k\wh W(k)}\eqqcolon K_* \Longrightarrow \text{$q_u$ is the unique global minimizer of $\mathcal{F}_K$} \Longrightarrow K_c \ge K_*.
\end{align}
Obviously, for $K_*>0$ to hold, some decay condition on $\wh{W}(k)$ is necessary. In order to ensure that $K_c=K_*$, the idea is to show that for any $K>K_*$, $q_u$ is not a global minimizer. For this strategy to work, we need $K_{\#}\le K_*$. These conditions are responsible for the decay condition \eqref{eq:decay condition}, and the equality for the first nonzero coefficient that is implied by the normalization $2\wh W(n+1)=1$ in \cref{thm:main pt}.

The uniqueness of the critical points is an elementary a posteriori consequence of the coercivity estimate established in the preceding discussion. But the argument has an inefficiency leading to the gap between uniqueness of critical points for $K\le 1/2$ and uniqueness of the global minimizer for $K\le 1$ discussed above.



\subsection{Applications to specific models}\label{ssec:intro-applications}

Having presented our main result \cref{thm:main pt} and commented on its proof, we now apply it to the motivating models \eqref{eq:Onsager kernel}--\eqref{eq: HK interaction} introduced in \cref{ssec:intro-examples}, thereby determining the exact value of $K_c$ and the (dis)continuity of the transition explicitly in terms of the model parameters. The theorems stated below are proven in \cref{sec:applications}.

The first result shows that the Doi--Onsager model exhibits a continuous phase transition at $K_{c} = 3\pi/4$, resolving a question left open by the previous works \cite{upperKonsager, upperonsager2, carrilloLongtimeBehaviourPhase2020}.

\begin{thm}\label{Doi-Onsager thm}
    In the Doi--Onsager model \eqref{eq:Onsager kernel}, a continuous phase transition occurs at
    \begin{equation}
    K_{c} = K_{\#} = \frac{3\pi}{4}.
    \end{equation}
\end{thm}

The second result for the noisy transformer model resolves the gap left open in the previous work of Balasubramanian et al. \cite{rigollet25} by showing that the previously introduced thresholds satisfy $\beta_-=\beta_+=\beta_*$.

\begin{thm} \label{transformer thm}
    For the noisy transformer model \eqref{transformer interaction}, let $\beta_{*}>0$ be the unique solution of $I_2(\beta) = \tfrac12 I_1(\beta)$. Then
    \begin{enumerate}[(1)]
        \item For $\beta\le\beta_\ast$, the phase transition is continuous and occurs at $K_c(\beta)=K_\#(\beta)$.
        \item For $\beta>\beta_\ast$, the phase transition is discontinuous and occurs strictly below the linear stability threshold, i.e.\ $K_c(\beta)<K_\#(\beta)$.
    \end{enumerate}
\end{thm}

Our final result improves upon the previous work \cite{carrilloLongtimeBehaviourPhase2020} for the Hegselmann--Krause model from the small-$R$ regime to the full range of $R \in [0, \pi]$.

\begin{thm} \label{HK thm}
    For the Hegselmann--Krause model \eqref{eq: HK interaction}, let $R_{*} \in (0, \pi)$ be the unique solution of $R = (\sin R)(2- \cos R)$. Then
    \begin{enumerate}[(1)]
        \item For $R < R_{\ast}$, the phase transition is discontinuous at $K_{c}(R) < K_\#(R)$.
        \item For $R \geq R_\ast$, the phase transition is continuous and occurs at $K_c(R)=K_\#(R)$.
    \end{enumerate}
\end{thm}

\subsection{Implications for the gradient flow}\label{ssec:intro-dynamics}

Since the McKean--Vlasov equation \eqref{MV equation} is the $2$-Wasserstein gradient flow of
$\mathcal F_K$, our static results for the free energy have direct consequences for
the long-time behavior of the dynamics. For the purposes of this discussion, we will continue to use the normalization of \cref{thm:main pt}, except when specializing to our motivating models. We further assume that $\partial_\theta W\in L^p$ for some $p>1$, which ensures global well-posedness of solutions for initial data in $\mathcal{P}(\T)$ as well as instantaneous $C^\infty$ smoothing, thereby avoiding any regularity issues. 

By standard compactness arguments, it is easy to show that any solution $q_t$ of \eqref{MV equation} converges weakly to a critical point $q_\infty$ of $\mathcal{F}_K$ as $t\rightarrow \infty$. The key question is what can be said about the selected equilibrium $q_\infty$ (e.g., is it the global minimizer?) and the rate of convergence. Uniqueness of critical points obviously answers the former question, which is guaranteed for $K\le 1/2$ by \cref{thm:main pt}. In the absence of uniqueness, the answer is more subtle and a priori depends on the initial data. 

Under the additional assumption that there exists a $C_W>0$ such that
\begin{align}
    |\partial_\theta^k W(\theta)| \le \frac{C_W^{k+1}k!}{|\theta|^{k-1}}, \quad \forall k\in \mathbb{N}, \ \theta \in \left[-\frac12, \frac12\right],
\end{align}
which is satisfied by the interactions in the motivating examples , the recent work of Choi et al. \cite{choiWassersteinLojasiewiczInequalitiesAsymptotics2025} implies that any critical point $q_\infty$ satisfies a \L{}ojasiewicz inequality and therefore the rate of decay $W_2(q_t,q_\infty)$ is at least algebraic, possibly exponential. Of course, this result does not tell us which critical point is selected by the flow or give the exact rate of convergence, but it does give some indication of the local geometry of the free-energy landscape around the selected equilibrium $q_\infty$. Nevertheless, in certain (local) regimes, one can say more about the selected equilibrium and the rate of convergence, as we now discuss.

In the subcritical regime $K<K_c$, for initial data $q_0$ sufficiently close to $q_u$, the solution $q_t$ converges to $q_u$ as $t\rightarrow \infty$ at an exponential rate governed by the spectral gap of the linearization around $q_u$,
\begin{align}
    \lambda_*(K):=\min_{k\ge1}\frac{k^2}{2}(1-K2\wh{W}(k))>0.
\end{align}
Consequently, one can show that
\begin{align}\label{eq:subcritical rate}
    W_2(q_t,q_u)\asymp e^{-\lambda_*(K)t} \quad \text{as } t\rightarrow \infty.
\end{align}
Under the additional assumption of uniqueness of critical points, one can say that given any initial data $q_0$, there is some time $T\ge 0$ such that for all $t\ge T$, the asymptotic \eqref{eq:subcritical rate} holds. 


If $K_c=K_{\#}$, then the spectral gap closes at criticality $K=K_c$, and the relaxation is expected to be algebraic, not exponential. Indeed, suppose that $\wh{W}$ has a unique dominant mode $k_* \ge 1$ in the sense that $\wh W(k_*) > \wh W(k)$ for all $k \ne k_*$.  Introducing the ratios $r_j:=\frac{\wh W(jk_*)}{\wh W(k_*)}$, preliminary calculations based on Landau-type free-energy expansions indicate the following predictions for initial data $q_0$ sufficiently close to $q_u$:
\begin{enumerate}
    \item (Quartic case) If $r_2<\frac12$, then $W_2(q_t,q_u)\asymp t^{-1/2}$ as $t\rightarrow\infty$.
    \item (Sextic case) If $r_2=\frac12$ and $r_3<\frac13$, then $W_2(q_t,q_u)\asymp t^{-1/4}$ as $t\rightarrow\infty$.
\end{enumerate}
 Similar to before, if one knows that $q_u$ is the unique critical point, then one can say that given any initial data $q_0$, there is some time $T\ge 0$ such that for all $t\ge T$, the above decay rates hold. Note that under the assumptions of \cref{thm:main pt}, the dominant mode is $k_*=n+1$ and the condition $r_2\le \frac12$  is satisfied. Proving these predicted asymptotics is the subject of ongoing work by the authors.

Based on these ans\"atz, we expect the following convergence rates for the models \eqref{eq:Onsager kernel}--\eqref{eq: HK interaction} at criticality $K=K_c$ for initial data sufficiently close to $q_u$:
\begin{enumerate}
    \item For the Doi--Onsager model, 
    \begin{align}
        W_2(q_t,q_u)\asymp t^{-1/2}.
    \end{align}
    \item For the noisy transformer model,
    \begin{align} \label{eq: transformer rate}
        W_2(q_t,q_u)\asymp \begin{cases}
            t^{-1/2}, & \quad \beta < \beta_* \\
            t^{-1/4}, & \quad \beta = \beta_*
        \end{cases}.
    \end{align}
    \item For the Hegselmann--Krause model,
    \begin{align} \label{eq: HK rate}
        W_2(q_t,q_u)\asymp \begin{cases}
            t^{-1/2}, & \quad R > R_* \\
            t^{-1/4}, & \quad R = R_*
        \end{cases}.
    \end{align}
\end{enumerate}
As in \eqref{eq: transformer rate} and \eqref{eq: HK rate}, the transition in the rate happens at the tricritical points $\beta=\beta_*$ and $R=R_*$, respectively, where $r_2 = \frac12$ and $r_k < \frac{1}{k}$ for $k \geq 3$.

The supercritical regime $K>K_c$ is the most subtle, as one is guaranteed to have at least two critical points, $q_u$ and a nonuniform global minimizer, and their respective (un)stable manifolds can in principle be quite complicated depending on the structure of $W$. At least in the case where $W$ has a sufficiently dominant single mode, the problem may be viewed as a perturbation of the unimodal case and one can prove a spectral gap for the linearization around a global minimizer. This then allows one to show local exponential convergence. We refer to our recent work \cite[Section 1]{mun_phase_2026} for more discussion on this point. Going beyond the perturbative regime is a difficult problem and the subject of ongoing research. 

\section{\texorpdfstring{Continuous phase transitions in multimodal models}{Continuous phase transitions in multimodal models}} \label{sec:contipt}
This section is devoted to the proof of \cref{thm:main pt}. We have already discussed the lower boundedness of the free energy and existence of minimizers in the introduction. We recall some preliminary facts about critical points and stability. 


Given $q\in\mathcal{P}(\T)$, consider a perturbation of the form $q+ \varepsilon \varphi$, where $\int_\T \varphi = 0$ so that $q + \varepsilon \varphi \in \mathcal{P}(\T)$ for sufficiently small $\varepsilon>0$. By Taylor's theorem,
\begin{equation}
    \mathcal{F}_{K}(q + \varepsilon \varphi) = \mathcal{F}_{K}(q) + \varepsilon \frac{\delta \mathcal{F}_K}{\delta q}(q) [ \varphi] + \frac{\varepsilon^{2}}{2} \frac{\delta^{2} \mathcal{F}_K}{\delta q^2}(q) [ \varphi, \varphi] + o(\varepsilon^2), \quad \varepsilon \to 0.
\end{equation}
The first and second variations are
\begin{align}
    &\frac{\delta \mathcal{F}_K}{\delta q}(q)[\varphi] = \int_\T \left(\log q +1 -2K (W \ast q) \right) \varphi d\theta, \\
    &\frac{\delta^{2} \mathcal{F}_K}{\delta q^2}(q)[\varphi, \varphi] = \int_\T \frac{\varphi^{2}}{q} d\theta - 2K \int_\T \varphi (W \ast \varphi) d\theta.  
\end{align} 
Evaluating at $q = q_u$, we obtain 
\begin{align}
    &\frac{\delta \mathcal{F}_K}{\delta q}(q_u)[\varphi] = \int_\T \varphi d\theta = 0,  \\
    &\frac{\delta^{2} \mathcal{F}_K}{\delta q^2}(q_u)[\varphi, \varphi] = 2 \sum_{k=1}^{\infty} \left(1- 2K\wh W(k)\right) |\wh\varphi(k)|^{2},
\end{align}
where the first identity follows from $\int_\T \varphi = 0$, and the second follows from Parseval's identity. Consequently, we observe the following:
\begin{enumerate}[1.]
\item $q_u$ is a critical point of $\mathcal{F}_K$ for all $K \in \mathbb{R}$,
\item $q_u$ is a local minimizer of $\mathcal{F}_{K}$ for $K \in [0, K_\#)$, where the linear stability threshold for $q_u$ is
    \begin{equation} \label{eq:Ksharp}
        K_\# \coloneqq \frac{1}{2\max_{k \in \cI} \wh W(k)}.
    \end{equation}
\item If $K > K_\#$, then there exists some $q \in \mathcal{P}(\T)$ such that
    \begin{equation} \label{eq: q_u loses stability}
        \mathcal{F}_K (q) < \mathcal{F}_K(q_u).
    \end{equation}
\end{enumerate}
Note that for $W$ satisfying the assumptions of \cref{thm:main pt}, we have $2\max_{k\in \cI} \wh W(k) =2\wh W(n+1)=1$ and $K_\# = 1$. 

Previous work \cite{chayesMcKeanVlasovEquationFinite2010,carrilloLongtimeBehaviourPhase2020} studied the continuity of phase transitions for general periodic interactions, relating the critical interaction strength $K_c$ and the (dis)continuity of the phase transition to the linear stability threshold $K_{\#}$ for $q_u$. We recall the following criterion from \cite[Proposition 2.12]{chayesMcKeanVlasovEquationFinite2010},\footnote{Strictly speaking, \cite{chayesMcKeanVlasovEquationFinite2010} assumed that the interaction $W\in L^\infty$ throughout their work, but this can be relaxed to our more general setting.} which will be used later below.

\begin{prop} \label{phase trans continuity prop}
Suppose $W$ is as above, and let $K_{c}, \ K_{\#}$ respectively be the phase transition threshold and linear stability threshold for $q_u$. Then the following hold:
\begin{enumerate}
    \item If $q_u$ is the unique global minimizer of $\mathcal{F}$ at $K=K_{\#}$, then $K_{c} = K_{\#}$ and the phase transition is continuous.
    \item \label{disconti criteria} If $q_{u}$ is not a global minimizer of $\mathcal{F}$ at $K=K_{\#}$, then $K_{c} < K_{\#}$ and the phase transition is discontinuous.
\end{enumerate}
\end{prop}

The preceding criterion does not provide a perfect dichotomy for determining the continuity of phase transitions since $q_u$ may still be a non-unique minimizer at $K=K_{\#}$. In our recent work \cite{mun_phase_2026}, we gave an example in which $K_c = K_\#$ but the phase transition is discontinuous.

\subsection{Entropy--interaction energy inequality}\label{ssec:contipt-entropy-ineq}
As advertised in the introduction, we use the constrained Lebedev--Milin inequality of \cref{prop:ops-widom} to prove the following sharp functional inequality that relates the entropy and the $\dot{H}^{-\frac12}$-seminorm. This inequality, which may be interpreted as the dual to \cref{prop:ops-widom}, is the key ingredient to the proof of \cref{thm:main pt} presented in the next subsection.

\begin{prop} \label{prop:sharp-ineq}
Let $n\ge 0$ be an integer. For every $\frac{1}{n+1}$-periodic $q\in\Pc_{ac}(\T)$, we have
\begin{equation}\label{eq:main-ineq}
\Hrel{q}{q_u}\ \ge\ (n+1)\sum_{k=1}^\infty |k|^{-1}|\wh{q}(k)|^2 = \pi(n+1)\|q\|_{\dot{H}^{-\frac12}}^2.
\end{equation}
Moreover, equality in \eqref{eq:main-ineq} holds if and only if
\begin{align}
    q(\theta) = q_{c,n}(\theta-\theta_0) \coloneqq \frac{1-c^2}{1+c^2-2c\cos(2\pi (n+1)(\theta-\theta_0))}, \quad c\in [0,1), \ \theta_0 \in \T.
\end{align}

\end{prop}

\begin{remark}
     Since $q_u$ is an extremizer of the inequality \eqref{eq:main-ineq}, it follows by considering perturbations $q_u+\varepsilon\cos(2\pi (n+1) \theta)$, letting $\varepsilon \to 0$, and using continuity that the constant $(n+1)$ is optimal. Also, one can remove the $\frac{1}{n+1}$-periodicity assumption on $q$ and instead replace $\wh{q}$ by $\wh{q} \indic_{k\in (n+1)\Z}$ in the right-hand side of \eqref{eq:main-ineq}. The two formulations are equivalent by the Fourier support condition.
\end{remark}

\begin{remark}
    If we allow both sides of the inequality \eqref{eq:main-ineq} to be infinite, then translates of the empirical measure $\frac{1}{n+1}\sum_{j=0}^n \delta_{\frac{j}{n+1}}$ are also extremizers. Moreover, these are the only non-absolutely continuous extremizers.
\end{remark}

\begin{proof}[Proof of \Cref{prop:sharp-ineq}]
By a standard density argument, it is sufficient to prove the result for $q \in \mathcal{P}(\T) \cap C^\infty(\T)$ that is $\frac{1}{n+1}$-periodic. We recall the dual formulation of the relative entropy (Donsker--Varadhan lemma):,
\begin{align} \label{eq: entropy-dual}
\Hrel{q}{q_u}
=
\sup_{\phi}\left\{
\int_\T \phi\,q\,d\theta
-\log\left(\int_\T e^{\phi}\,d\theta\right)
\right\},
\end{align}
where the supremum is taken over all bounded measurable functions $\phi$.
Let $g \coloneqq q - q_u$, and for $t\in\R$ to be chosen, define the test function $\phi$ by $\wh{\phi}(k):= t |k|^{-1}\wh{g}(k)$. Evidently, $\phi$ is $\frac{1}{n+1}$-periodic, and therefore so is $e^{\phi}$, implying the constraint is satisfied. Using \eqref{eq: entropy-dual}, we obtain
\begin{align}
    \Hrel{q}{q_u} &\ge \int_\T \phi\,q\,d\theta - \log\int_\T e^{\phi}\,d\theta \nonumber \\
    &\ge \int_{\T} \phi\,g\,d\theta - \frac{1}{n+1} \sum_{k=1}^\infty |k| |\wh{\phi}(k)|^{2} \nonumber \\
    &= 2t\sum_{k=1}^\infty |k|^{-1}|\wh{g}(k)|^2 - \frac{t^2}{n+1}\sum_{k=1}^\infty |k|^{-1}|\wh{g}(k)|^{2}\nonumber\\
    &\ge \left(2t-\frac{t^2}{n+1}\right)\sum_{k=1}^\infty |k|^{-1}|\wh{g}(k)|^2 , \label{eq:main-ineq-2}
\end{align}
 where the second line follows from \cref{prop:ops-widom} and \cref{rem:H12} and the third follows from Parseval's identity. The function $t\mapsto 2t-\frac{t^2}{n+1}$ is maximized at $t={n+1}$. Substituting this choice into \eqref{eq:main-ineq-2}, we obtain the desired inequality.
 

If equality in \eqref{eq:main-ineq} holds, then the test function $\phi$ defined above is an extremizer of \cref{prop:ops-widom}.  Otherwise, the second line of \eqref{eq:main-ineq-2} would have a strict inequality. Therefore, $e^{-\phi}$ is a trigonometric polynomial of degree at most $n+1$. The requirement that $\int_{\T}e^{-\phi}e^{2\pi i k \theta}d\theta=0$ for all $1 \leq k \leq n$ implies that
\begin{align}
    e^{-\phi} =a_0+ a_{n+1} e^{2\pi i (n+1) \theta} + \bar{a}_{n+1} e^{-2\pi i (n+1) \theta}
\end{align}
for some $a_0>0$ and $a_{n+1} \in \mathbb{C}$. Since $\phi$ is finite, real-valued, we have $a_0 + 2 \mathrm{Re}(a_{n+1} e^{2\pi i (n+1) \theta}) > 0$ for all $\theta \in \T$. This implies that $2|a_{n+1}| <  a_0$. Translating $\theta$, we can assume without loss of generality that $a_{n+1} = |a_{n+1}|$. Since $\frac{2|a_{n+1}|}{a_0} < 1$ and $c\mapsto \frac{2c}{1+c^2}$ is a bijection from $[0,1)$ to $[0,1)$, we can write $\frac{2|a_{n+1}|}{a_0} = \frac{2c}{1+c^2}$ for a unique $c \in [0,1)$. On the other hand, equality holds in \eqref{eq: entropy-dual} if and only if $q=\frac{e^{\phi}}{\int_\T e^{\phi}\,d\theta}$. Hence,
\begin{align}
    q(\theta) \propto \frac{1}{1+c^2-2c\cos(2\pi (n+1) \theta)}.
\end{align}
Using that $\int_{\T}q\,d\theta=1$, we find that $q =\frac{1-c^2}{1+c^2-2c\cos(2\pi (n+1) \theta)}$

Conversely, if $q=\frac{1-c^2}{1+c^2-2c\cos(2\pi (n+1) \theta)}$ for some $c \in [0,1)$, then
\begin{align}
    &q(\theta)-1= \sum_{k\ne 0} c^{|k|} e^{2\pi i k (n+1) \theta}, \\
    &\phi(\theta) := -\log\left(1+c^2-2c\cos(2\pi (n+1) \theta)\right) = \sum_{k\ne 0}\frac{c^{|k|}}{|k|} e^{2\pi i k (n+1) \theta}.
\end{align}
Hence,
\begin{align}
    \|q-q_u\|_{\dot{H}^{-\frac12}}^2 &= \sum_{k\ne 0} |2\pi(n+1)k|^{-1} c^{2k} = \frac{1}{\pi(n+1)}\sum_{k=1}^\infty \frac{c^{2k}}{k} = -\frac{\log(1-c^2)}{\pi(n+1)},
\end{align}
and similarly,
\begin{align}
    \int_{\T}\log(q) q\,d\theta &= \int_{\T} \phi q\,d\theta  + \log(1-c^2)= \log(1-c^2) + \sum_{k\ne 0} \frac{c^{2|k|}}{|k|} = -\log(1-c^2).
\end{align}
This implies that $q$ is an extremizer of \eqref{eq: entropy-dual}, thereby completing the proof.

\end{proof}

\subsection{\texorpdfstring{Proof of \cref{thm:main pt}}{Proof of the main theorem}}\label{ssec:contipt-proof}

With \cref{prop:sharp-ineq} in hand, we now prove \cref{thm:main pt}.

Let $q$ be a global minimizer of $\mathcal{F}_K$. First, we claim that $q$ is $\frac{1}{n+1}$-periodic. Indeed, any critical point of $\mathcal{F}_K$ satisfies the Kirkwood--Monroe equation
\begin{equation}
    q(\theta) = \frac{1}{Z} e^{2K (W \ast q)(\theta)}, \quad Z = \int_\T e^{2K (W \ast q)(\theta)} d\theta.
\end{equation}
Since $W$ is $\frac{1}{n+1}$-periodic, the convolution $W \ast q$ is also $\frac{1}{n+1}$-periodic, which implies the claim.


Next, we write the free energy difference as
\begin{align}\label{eq:free-energy-diff}
    \mathcal{F}_K(q) - \mathcal{F}_K(q_u) &= \Hrel{q}{q_u} - K \int_\T (W \ast (q-q_u)) (q-q_u) d\theta \nonumber \\
    &= \left(\Hrel{q}{q_u} -(n+1)\sum_{k=1}^\infty |k|^{-1}|\wh{q}(k)|^2\right) + \sum_{k=1}^\infty \left((n+1)|k|^{-1}-2K\wh{W}(k)\right)|\wh{q}(k)|^2.
\end{align}
 By \Cref{prop:sharp-ineq}, the first term on the right-hand side is nonnegative with equality if and only if (modulo translation) $q=q_{c,n}$ for $c\in [0,1)$. For the second term, let
\begin{align}\label{eq:K*def}
    K_* &\coloneqq \sup\left\{K \geq 0: (n+1)|k|^{-1}-2K\wh{W}(k) \geq 0, \ \forall k\ge 1\right\} \nonumber \\
    &= \sup\left\{K\ge0:\frac{1}{\ell}-2K\wh W((n+1)\ell)\ge0,\ \forall \ell\ge 1\right\}
\end{align}
where the second line follows from the assumption $W$ is $\frac{1}{n+1}$-periodic. 
Tautologically, for any $K\le K_*$, the second term on the right-hand side of \eqref{eq:free-energy-diff} is nonnegative. If $q=q_{c,n}$ for $c\in (0,1)$, then since $|\wh{q}_{c,n} (\ell(n+1))|>0$ for all $\ell\ge 1$, we have by \cref{rem:decay-condition} a strict inequality for some $k = \ell(n+1)$ in the second term of \eqref{eq:free-energy-diff}, implying $q_{c,n}$ is not a global minimizer. Hence, $q_u$ is the unique global minimizer of $\mathcal{F}_K$ for $K \leq K_*$. Thus, $K_*\le K_c$. 

We next identify $K_*$. The decay condition implies $2\wh W((n+1)\ell)\le \ell^{-1}$ for every $\ell\in\mathbb N$, and for $\ell=1$ we have equality. Hence $K_*=1 = K_\#$. Since $K_c \leq K_\#$, we conclude that $K_c = K_* = K_\#=1$ and the phase transition is continuous.

Finally, suppose $q_*$ is a critical point of $\mathcal{F}_K$. Then the Kirkwood--Monroe equation implies the relation
\begin{align}
    \mathcal{F}_K(q) - \mathcal{F}_K(q_*) = \Hrel{q}{q_*} - K \int_\T (W \ast (q-q_*)) (q-q_*) d\theta.
\end{align}
Taking $q=q_u$ in this relation and combining with the first line of \eqref{eq:free-energy-diff} (with $q=q_*$), we obtain
\begin{align}
    \Hrel{q_*}{q_u}  + \Hrel{q_u}{q_*} = 2K \int_\T (W \ast (q_u-q_*)) (q_u-q_*) d\theta.
\end{align}
On the other hand, using the nonnegativity of the relative entropy, it follows from \cref{prop:sharp-ineq} and the decay condition \eqref{eq:decay condition} that
\begin{align}
     \int_\T (W \ast (q_u-q_*)) (q_u-q_*) d\theta\le \Hrel{q_*}{q_u} &\le \Hrel{q_*}{q_u}  + \Hrel{q_u}{q_*} \nonumber\\
     &= 2K \int_\T (W \ast (q_u-q_*)) (q_u-q_*) d\theta.
\end{align}
If $2K<1$, then it is obvious that $q_*=q_u$. If $2K=1$, then the above inequalities must be equalities, which implies that $H(q_u\vert q_*)=0$. Hence, $q_u=q_*$. This completes the proof.  


\section{Applications to mean-field models} \label{sec:applications}
In this section, we apply \cref{thm:main pt} to the three models \eqref{eq:Onsager kernel}, \eqref{transformer interaction}, \eqref{eq: HK interaction} from \cref{ssec:intro-examples}. 

\subsection{Doi--Onsager Model}\label{ssec:applications-DO}
We first prove \cref{Doi-Onsager thm} for the two-dimensional Doi--Onsager model for which $K_\# = (2\wh W(2))^{-1}$.



Evidently, $W_{\DO}$ is $\frac12$-periodic, and we recall that
\begin{align}
    \wh W_{\DO}(2\ell)=\frac{2}{\pi}\frac{1}{4\ell^2-1}, \quad \ell \geq 1.
\end{align}
For $\ell\ge2$, we have
\begin{align}
    \wh W_{\DO}(2\ell)<\frac{\wh W_{\DO}(2)}{\ell}
    \iff
    \frac{1}{4\ell^2-1}<\frac{1}{3\ell}
    \iff
    (4\ell+1)(\ell-1)>0.
\end{align}
Therefore, the rescaled interaction potential ${W_{\DO}}/{(2\wh W_{\DO}(2))}$ satisfies the assumptions of \cref{thm:main pt} with $n=1$, and the conclusion follows.

\subsection{Noisy Transformer Model}
Next, we prove \cref{transformer thm} for the noisy transformer model. 

Observe from \eqref{transformer interaction} that the nonzero Fourier coefficients of $W_\beta$ are
\begin{equation}
2\wh W_\beta(\ell)=\frac{2}{\beta}I_\ell(\beta), \quad \ell\ge1.
\end{equation}
Since for each fixed $\beta$, the map $\ell \mapsto I_\ell(\beta)$ is decreasing on $\mathbb{N}$  (see \cite{soni1965inequality}) , we have $I_\ell(\beta) \le I_1(\beta)$ for all $\ell \ge 1$. Hence,
\begin{equation}
K_\#(\beta)=\frac{1}{2\wh W_\beta(1)}=\frac{\beta}{2I_1(\beta)}.
\end{equation}
Next, observe that $I_{2}(x) - \frac{I_1(x)}{2}$ is convex for $x>0$ and is $0$ at $x=0$. We therefore define $\beta_\ast \approx 2.447$ to be the unique positive solution of
\begin{equation}\label{eq:beta-star}
    I_2(\beta_\ast)-\frac{I_1(\beta_\ast)}{2}=0.
\end{equation}
The next lemma then shows that $\frac{\wh W_\beta(\ell)}{\wh W_\beta(1)} \leq \frac{1}{\ell}$ for all $\ell \ge 1$ whenever $\beta\le\beta_\ast$.

\begin{lemma} \label{lemma: transformer}
If $\beta\le\beta_\ast$, then
\begin{equation} \label{eq:Bessel decay}
I_\ell(\beta)\le \frac{I_1(\beta)}{\ell}, \qquad\forall \ell\ge1,
\end{equation}
and the inequality is strict for all $\ell\ge3$.
\end{lemma}

\begin{proof}
We prove the lemma by induction on $\ell$. The case $\ell=1$ is trivial, and the case $\ell=2$ follows directly from the definition of $\beta_*$.

For the induction step, suppose that \eqref{eq:Bessel decay} holds for some $\ell \ge 2$. Recalling the expansion \cite[p. 77]{watson1922treatise} 
\begin{equation}
I_\ell(\beta)=\sum_{k=0}^\infty \frac{1}{k!(k+\ell)!}\left(\frac{\beta}{2}\right)^{2k+\ell},
\end{equation}
we see that
\begin{align}
I_{\ell+1}(\beta)=\sum_{k=0}^\infty \frac{1}{k!(k+\ell+1)!}\left(\frac{\beta}{2}\right)^{2k+\ell+1}
&\le \frac{\beta}{2(\ell+1)}\sum_{k=0}^\infty \frac{1}{k!(k+\ell)!}\left(\frac{\beta}{2}\right)^{2k+\ell} \nonumber\\
&=\frac{\beta}{2(\ell+1)}I_\ell(\beta) \le \frac{\beta}{2\ell}\frac{I_1(\beta)}{\ell+1},
\end{align}
where the last inequality follows from the induction hypothesis. Since $\beta\le\beta_\ast<4$ and $\ell\ge 2$, we have $\frac{\beta}{2\ell} < 1$, which completes the proof of the induction step and hence the lemma.
\end{proof}

If $\beta>\beta_\ast$, then by the definition of $\beta_*$, we have $\wh W_\beta(2)>\frac{\wh W_\beta(1)}{2}$. Thus, $\frac{W_{\beta}}{2\wh W_\beta(1)}$ satisfies the decay condition \eqref{eq:decay condition} if and only if $\beta\le\beta_\ast$. The following general lemma implies that the phase transition is discontinuous when $\beta>\beta_*$. The argument, based on a choosing a suitable perturbation around $q_u$, is similar to the bimodal case studied by the authors in \cite{mun_phase_2026}.

\begin{lemma}\label{lemma:bimodal}
If the interaction potential $W$ satisfies $2\wh W(1)=1$, $2\wh W(2) > 1/2$, and
\begin{align}
 2\wh W(k) \leq 1, \quad \forall k \ge 1,
\end{align}
then the phase transition is discontinuous at $K_c< K_\# = 1$.
\end{lemma}
\begin{proof}
    First, note from \eqref{eq:Ksharp} that $K_\# = 1$. For $0<\varepsilon\ll 1$, consider
    \begin{align}
        q_{\varepsilon}(\theta) = q_u(1+ {\varepsilon} \cos (2\pi\theta) + c\varepsilon^{2} \cos (4\pi\theta)),
    \end{align}
    where $c \in \mathbb{R}$ will be chosen momentarily. A direct expansion gives 
    \begin{align}
        \Hrel{q_{\varepsilon}}{q_u} = \frac14 \varepsilon^2 + \left( \frac{c^2}{4} - \frac{c}{8} + \frac{1}{32}\right)\varepsilon^4 + O(\varepsilon^5),\quad \int_\T (W \ast q_{\varepsilon}) \, q_{\varepsilon} \, d\theta = \frac{1}{4} \varepsilon^2 + \frac{2\wh W(2)}{4}c^2\varepsilon^4.
    \end{align}
    Therefore, at $K = 1$,
\begin{equation}
    \mathcal{F}_{K}(q_{\varepsilon}) - \mathcal{F}_{K}(q_u) = \underbrace{\left( \frac14 \left(1 - 2\wh W(2)\right)c^{2} - \frac{c}{8} + \frac{1}{32}\right)}_{\eqqcolon p(c)}\varepsilon^4 + O(\varepsilon^5).
\end{equation}
     If $2\wh W(2)=1$, then it is obvious that we may choose $c$ so that $p(c)<0$. If $2\wh W(2)<1$, then $p$ has a global minimizer at $c_* = \frac{1}{4(1-2\wh W(2))} > 0$ and $p(c_*) = \frac{1-4\wh W(2)}{64(1-2\wh W(2))} < 0$ by our assumption that $4\wh{W}(2)>1$. Choosing $p(c)<0$, it now follows that $\mathcal{F}_K(q_\varepsilon) - \mathcal{F}_K(q_u) <0$ for sufficiently small $\varepsilon>0$. Thus, $q_u$ is not a global minimizer of $\mathcal{F}_{K}$ at $K= K_\#$, and the discontinuity of the phase transition follows from \cref{phase trans continuity prop}\eqref{disconti criteria}.
\end{proof}

We now complete the proof of \cref{transformer thm} by combining \Cref{lemma: transformer,lemma:bimodal}.
\begin{proof}[Proof of \cref{transformer thm}]
If $\beta\le\beta_\ast$, then \cref{lemma: transformer} shows that the normalized coefficients satisfy the decay condition \eqref{eq:decay condition}, so \cref{thm:main pt} and the rescaling remark after \eqref{eq: W-multimodal} yield that the phase transition is continuous and satisfies $K_c(\beta)=K_\#(\beta)$. If instead $\beta>\beta_\ast$, then the normalized coefficients satisfy the assumptions of \cref{lemma:bimodal}, and hence the phase transition is discontinuous at $K_c(\beta)<K_\#(\beta)$.
\end{proof}

\subsection{Hegselmann--Krause Model}
Finally, we prove \cref{HK thm} for the Hegselmann--Krause model.

From \eqref{eq: HK interaction}, we recall that 
\begin{align}
\wh W_R(\ell) = \frac{2}{\pi \ell^3} (\ell R - \sin \ell R), \qquad \ell \ge 1.
\end{align}
Since $g(x) \coloneqq x-\sin x$ is strictly positive for $x>0$, it follows that $\wh W_R(\ell) >0$ for all $\ell\ge 1$. The next lemma shows that $\wh W_R$ is maximized at $\ell=1$, and hence that $K_\#(R) = (2\wh W_R(1))^{-1}$.

\begin{lemma} \label{lemma: HK monotone}
For every $R>0$ and $\ell\ge1$, we have 
    $\wh W_R(\ell) \le \wh W_R(1)$.
\end{lemma}
\begin{proof}
    Write $\wh W_R(\ell) = \frac{4}{\pi} \, \frac{g(\ell R)}{\ell^{3}}$. Since $g(x) = \int_{0}^{x} (1- \cos t) \, dt$, we obtain
    \begin{align}
        g(\ell R) = \ell \int_0^{R} (1- \cos \ell t) \, dt = 2\ell \int_0^{R} \sin^{2} \frac{\ell t}{2} \, dt 
        &\leq 2\ell^{3} \int_0^{R} \sin^{2} \frac{t}{2} \, dt = \ell^{3} g(R),
    \end{align}
    where we used the elementary inequality $|\sin(mu)| \leq m|\sin u|$ for any integer $m\ge 0$. The desired conclusion now follows from reexpressing $\wh W_R$ in terms of $g$.
\end{proof}

The function $\varphi(R):=R - (\sin R)(2-\cos R)$ satisfies $\varphi(0)=0$, has a global minimizer on $[0, \pi]$ at $R=\pi/2$, and satisfies $\varphi'(R)<0$ for $R\in (0,\pi/2)$ and $\varphi'(R)>0$ for $R\in (\pi/2, \pi)$; therefore, it follows from the intermediate value theorem that $\varphi$ has a unique zero $R_* \approx 2.139$ in $(0, \pi)$. Now observe that
\begin{align}
    \frac{\wh W_R(2)}{\wh W_R(1)} = \frac{1}{2} \iff R = (\sin R)(2-\cos R).
\end{align}
For $R < R_*$, we have $\wh W_R(2)/\wh W_R(1) > \frac{1}{2}$, and thus \cref{lemma:bimodal} applied to the normalized interaction potential $W_{R}/(2\wh W_R(1))$ implies that the phase transition is discontinuous at $K_c(R) < K_\#(R)$. It remains to show that the normalized coefficients satisfy the decay condition \eqref{eq:decay condition} when $R \geq R_*$. This will then complete the proof of \cref{HK thm}.

\begin{lemma}\label{lemma: HK decay}
If $R \geq R_*$, then
\begin{align}
\wh W_R(\ell) \leq \frac{\wh W_R(1)}{\ell}, \quad \forall \ell \geq 1,
\end{align}
and the inequality is strict for every $\ell \geq 3$.
\end{lemma}
\begin{proof}[Proof of \cref{lemma: HK decay}]
We prove the lemma by induction on $\ell$. The cases $\ell=1,2$ are immediate. Assume $R \geq R_*$, and observe that
\begin{align}
2\wh W_R(\ell) \leq \frac{2\wh W_R(1)}{\ell} &\iff g(\ell R) \leq \ell^2 g(R)\iff \Phi_\ell(R) \coloneqq \ell^2 g(R) - g(\ell R) \geq 0.
\end{align}

For the induction step, suppose that $\Phi_\ell(R) \geq 0$ for some $\ell \geq 2$. Then,
\begin{align} \label{eq:Phi diff}
\Phi_{\ell+1}(R) - \Phi_\ell(R) &= (2\ell +1)g(R) - R + \bigl(\sin((\ell+1)R) - \sin (\ell R)\bigr) \nonumber\\
&\geq (2\ell +1)g(R) - R  - 2 \sin \frac{R}{2} \nonumber\\
&\geq 5g(R) - R  - 2 \sin \frac{R}{2},
\end{align}
where the first inequality follows from
\begin{align}
    \sin((\ell+1)R) - \sin (\ell R) = 2 \cos((\ell + 1/2)R) \sin\frac{R}{2} \geq -2 \sin \frac{R}{2},
\end{align}
and the second inequality follows from $\ell \geq 2$. Setting $J(R) \coloneqq 5g(R) - R - 2 \sin(R/2)$, a direct computation yields
\begin{equation}
 J'(R) =\left(9-10\cos\frac{R}{2}\right)\left(1+\cos\frac{R}{2}\right).
\end{equation}
Since $(R_*, \pi) \subset (\frac\pi2, \pi)$, we have $\cos(R/2) < 9/10$, and hence $J'(R)>0$ on $(R_*, \pi)$. Therefore, $J(R) > J(0) = 0$, and hence $\Phi_{\ell+1}(R) \geq \Phi_\ell(R)$. This proves the claimed statement. The strict inequality for $\ell \geq 3$ follows from the strict monotonicity in the argument above.
\end{proof}

\bibliographystyle{alpha}
\bibliography{References}

\end{document}